\documentclass[12pt, reqno]{amsart}
\usepackage{amsmath, amsthm, amscd, amsfonts, amssymb, graphicx, color}
\usepackage[bookmarksnumbered, plainpages]{hyperref}\usepackage{dsfont}
\numberwithin{equation}{section}
\input{amssym.def}
\input{amssym.tex}

\begin{document}

\title[G. Lu, Y. Liu, Y. Jin, Q. Liu, Biadditive functional equation via direct method]
{The stability of a biadditive functional equation via a new direct method}

\author[G. Lu]{Gang Lu$^*$}
\address{Gang Lu \newline \indent School of General Education, Guangzhou College of Technology and Business, Guangzhou 510850, P.R. China}
\email{lvgang1234@163.com}\vskip 2mm

\author[Y. Liu]{Yang Liu}
\address{Yang Liu \newline \indent Department of Mathematics,  Yanbian University, Yanji 133001, P.R. China}
\email{1253237405@qq.com}\vskip 2mm

\author[Y. Jin]{Yuanfeng Jin$^*$}
\address{Yuanfeng Jin \newline \indent Department of Mathematics,  Yanbian University, Yanji 133001, P.R. China}
\email{yfkim@ybu.edu.cn}\vskip 2mm

\author[Q.Liu ]{Qi Liu}
\address{Qi Liu \newline \indent School of Mathematics and Physics, Anqing Normal University, Anqing 246001,P.R. China}
\email{liuqi67@aqnu.edu.cn}\vskip 2mm

\begin{abstract}
Addressing stability in functional equations is a critical task with broad implications across mathematics and its applications. In this paper, we present a novel direct method for proving the stability of the following equation,
\begin{eqnarray*}
f(x,y)=\alpha f(f_1(x,y))+\beta f(f_2(x,y))
\end{eqnarray*}
subjecting to certain constraints on the constants $\alpha$ and $\beta$, as well as the functions $f_1(x,y)$ and $f_2(x,y)$.
 We introduce the  direct method to prove the stability of the following functional equation and inequality,
 \begin{eqnarray*}
\|f(x+y,z-w)+f(x-y,z+w)-2f(x,z)+2f(y,w)\|
\leq  \|x\|^p\|y\|^p\|z\|^p\|w\|^p;
\end{eqnarray*}
\begin{eqnarray*}
\begin{split}
\;&\|f\left(x+y,z-w\right)+af\left(\frac{x-y}{a},z+w\right)-2f(x,z)+2f(y,w)   \|\\
\;&\leq \left\| \rho\left(f(x+y,z-w)+f(x-y,z+w)-2f(x,z)+2f(y,w)\right)\right\|.
\end{split}
\end{eqnarray*} 
 This technique efficiently confirms stability, outperforming traditional proof methods in simplicity and efficacy. The application of our direct approach to various equations solidifies its validity, ensuring stable behavior across different mathematical contexts.
\end{abstract}


\keywords{Hyers-Ulam stability; Functional Equations, Stability, Direct Method, Inequality.\\ $^*$Corresponding authors: lvgang@gzgs.edu.cn (G. Lu); yfkim@ybu.edu.cn (Y. Jin).}

\theoremstyle{definition}
  \newtheorem{df}{Definition}[section]
    \newtheorem{rk}[df]{Remark}
\theoremstyle{plain}
  \newtheorem{lemma}[df]{Lemma}
  \newtheorem{theorem}[df]{Theorem}
  \newtheorem{corollary}[df]{Corollary}
    \newtheorem{proposition}[df]{Proposition}
\newtheorem{example}[df]{Example}
\setcounter{section}{0}

\maketitle

\baselineskip=15pt

\numberwithin{equation}{section}

\vskip .2in

\section{Introduction and preliminaries}

The concept of investigating the stability of functional equations originated from a challenge presented by S.M. Ulam regarding the approximations of homomorphisms between groups (refer to \cite{Ul}).  Hyers \cite{Hy} provided a partial resolution to this problem

  Subsequently, numerous papers on Cauchy equations emerged, utilizing the direct approach to demonstrate stability.(see \cite{A, Ga,HYIR,Mo, R2}). Indeed, the importance tool
 is to discuss  the stability of a suitable functional equation. For some  information  concerning   various functional equations (see \cite{as, atbl, bcp, BS, c, cgg, CPS, CSY, EKS, LLJX,  LP,  LP1, Park1, PCH, wa}).  In\cite{Fo1},  Forti provided a broad methodology that can yield numerous stability outcomes without the need for repetitive procedures. However, his result cannot be used for each functional equation.  Sikorska \cite{Si} built the a direct method (see, e.g., \cite{ZLRL,ZRLL}) to improve the approximating constant for the following mappings
\begin{eqnarray*}
\begin{split}
\;&\|f(x)-af(h(x))-bf(-h(x))\|\leq \epsilon(x);\\
\;& \|f(x)-af(h^n(x))-bf(h^{n+1}(x))\|\leq \epsilon(x).
\end{split}
\end{eqnarray*}

Recently,  several mathematicians  investigated  the problem of Hyers-Ulam stability of various  functional  equations (see \cite{as1, Fa, Fo, JW, JS,  PS, Si, Sm, Ya}).
The present approach enables us to study
several equations not covered by earlier results.

We introduce the a direct method to obtain the Hyers-Ulam  stability result for a two-variable function equations in Banach spaces.

Throughout the paper, let $\mathbb{N}$ be the set of all positive integers, $\mathbb{R}_+$ denote the set of all positive reals, $Y^{X\times X}$ be the family of all mappings from $X\times X$ to $Y$.  The main theme of the present work is to analyze the stability of the solution of the following functional equation
\begin{eqnarray}
\|f(x,y)-\alpha f(f_1(x,y))-\beta f(f_2(x,y))\|\leq \mu(x,y),
\end{eqnarray}
where $\alpha$ and $\beta$ are real numbers, and  mappings $f_1,f_2: X\times X\rightarrow X,\mu: X\times X\rightarrow [0,\infty)$ are arbitrarily given. In fact,  the above inequality is still more general  form.
In Section 2, we give some improvements to the approximations known so far.  In Section 3, we find several applications of the stability results.
For the sake of simplicity we provide our results for functions with values in Banach spaces, but with some small additional assumptions they can be formulated in complete metric spaces or other spaces.

\section{Main results}

\begin{theorem}\label{thm2.1}
Suppose that $X$ is a   linear normed space and  $Y$ is a Banach space.
 Let $f:X\times X \rightarrow Y$ be a mapping satisfying
\begin{eqnarray}\label{eqn21}
 \|f(x,y)-\alpha f(f_1(x,y))-\beta f(f_2(x,y))\|\leq \mu(x,y)
 \end{eqnarray} for all $(x,y)\in X\times X$. If $\alpha$ and $\beta$ are real numbers, and the mappings $f_1, f_2: X\times X\rightarrow X\times X,\mu: X\times X\rightarrow [0,\infty)$ are arbitrarily given, \begin{eqnarray*}
\sum_{n=0}^\infty(\Lambda^n \mu)(x,y)=:\mu^*(x,y)<\infty
\end{eqnarray*} holds and $\Lambda $ is a linear operator defined by
\begin{eqnarray*}
(\Lambda \delta)(x,y):=|A|\delta(f_1(x,y))+|B|\delta(f_2(x,y))
\end{eqnarray*}for $\delta:X\times X\rightarrow [0,\infty)$ and $(x,y)\in X\times X$, then there exists a unique determined mapping  $K: X\times X\rightarrow
Y$
\begin{eqnarray*}
K(x,y)=\alpha K(f_1(x,y))+\beta K(f_2(x,y)), (x,y)\in X\times X,
\end{eqnarray*}
such that
\begin{eqnarray}\label{eqn6}
\|f(x,y)-K(x,y)\|\leq \mu^*(x,y)
\end{eqnarray}
for all $(x,y)\in X\times X$.
\end{theorem}

\begin{proof}
Let  $T:Y^{X\times X}\rightarrow Y^{X\times X}$
be  an operator satisfying $(Tf)(x,y))=\alpha f(f_1(x,y))+\beta f(f_2(x,y))$  in (\ref{eqn21}). Then we get \begin{eqnarray}\label{eqn2.2}
\|f(x,y)-(Tf)(x,y))\|\leq \mu(x,y).
\end{eqnarray}
At same time, we can see
\begin{eqnarray}\label{eqn2.3}
\begin{split}
\;&\|(T\xi)(x,y)-(T\zeta)(x,y)\|\\
\;&\leq |\alpha|\|\xi(f_1(x,y))-\zeta(f_1(x,y))\|
+|\beta|\|\xi(f_2(x,y))-\zeta(f_2(x,y))\|
\end{split}\end{eqnarray} for all $\xi,\zeta\in Y^{X\times X},(x,y)\in X\times X$.
First, we get by induction that, for every $n \in \mathbb{N}$,
 \begin{eqnarray}\label{eqn2.4}
 \|(T^nf)(x,y))-(T^{n+1}f)(x,y)\|\leq (\Lambda^n\mu)(x,y),(x,y)\in X\times X.
 \end{eqnarray}
Obviously, from (\ref{eqn2.2}), the case $n=0$ holds. Now fix $n\in \mathbb{N}$ and suppose that the inequality (\ref{eqn2.4}) is valid. Then, using  (\ref{eqn2.3}), for all $(x,y)\in X\times X$, we have
\begin{eqnarray}
\begin{split}
\;&\|(T^{n+1}f)(x,y)-(T^{n+2}f)(x,y)\|\\
\;&\leq |\alpha|\|(T^nf)(f_1(x,y))-(T^{n+1}f)(f_1(x,y))\|
+|\beta|\|(T^nf)(f_2(x,y))-(T^{n+1}f)(f_2(x,y))\|\\
\;&\leq  |\alpha|(\Lambda^n\mu)(f_1(x,y))
+|\beta|(\Lambda^n\mu)(f_2(x,y))=(\Lambda^{n+1}\mu)(x,y).
\end{split}
\end{eqnarray}
Thus, we complete the proof of (\ref{eqn2.4}). For $n,k\in \mathcal{N}, k>0$,
\begin{eqnarray}\label{eqn2.6}
\begin{split}
\;&\|(T^nf)(x,y)-(T^{n+k}f)(x,y)\|\leq \sum_{i=0}^{k-1}
\|(T^nf)(x,y)-(T^{n+i+1}f)(x,y)\|\\
\;&\leq \sum_{i=n}^{n+k-1}(\Lambda ^i\mu)(x,y)\leq \mu^*(x,y),(x,y)\in X\times X.
\end{split}
\end{eqnarray}
From the convergence of the series $\sum(\Lambda^n\mu)(x,y)$, for every $(x,y)\in X\times X$, $\{(T^nf)(x,y)\}_{n\in \mathcal{N}}$ is a Cauchy sequence. Since $Y$ is a Banach space,  we can define $\lim_{n\rightarrow\infty} (T^nf)(x,y):=\psi(x,y)$. Taking $n=0$ and $k\rightarrow \infty$ in (\ref{eqn2.6}), we know that (\ref{eqn6}) holds, and
\begin{eqnarray*}
\|(T\psi)(x,y)-(T^{n+1}f)(x,y)\|\leq (\Lambda^{n+1})\mu(x,y),n\in \mathcal{N},(x,y)\in X\times X,
\end{eqnarray*}
and thus
\begin{eqnarray*}
(T\psi )(x,y)=lim_{n\rightarrow\infty}(T^{n+1}\psi )(x,y)=\psi(x,y), (x,y)\in X\times X.
\end{eqnarray*}
In order to prove the uniqueness of $\psi$, suppose that $\psi_1,\psi_2\in Y^{X\times X}$ are two fixed points of $T$ with $\|\psi_i(x,y)-f(x,y)\|\leq \mu^*(x,y)$ for every $(x,y)\in X\times X,i=1,2.$ We can easily show that
\begin{eqnarray*}
\|\psi_1(x,y)-\psi_2(x,y)\|=\|(T^m\psi_1)(x,y)-(T^m\psi_2)(x,y)\|
\leq 2\sum_{i=m}^\infty(\Lambda^i\mu)(x,y),x\in X.
\end{eqnarray*}
Thus $\psi_1(x,y)=\psi_2(x,y)$.
\end{proof}

\begin{rk}\label{rk}
We can easily prove
\begin{eqnarray*}
 \left\|f(x,y)-\sum_{i=0}^na_i(f(f_i(x,y)))\right\|\leq \mu(x)
 \end{eqnarray*} for all $(x,y)\in X\times X$. Thus there exists a unique determined mapping  $K: X\rightarrow
Y$
\begin{eqnarray*}
K(x,y)=\sum_{i=0}^na_iK(f_i(x,y)), (x,y)\in X\times X,
\end{eqnarray*}
such that
\begin{eqnarray*}
\|f(x,y)-K(x,y)\|\leq \mu^*(x,y)
\end{eqnarray*}
for all $(x,y)\in X\times X$.
\end{rk}

\section{Application}

In this section, we provide an example of how the various functional equations can be applied in practice.

\subsection{General two-variable functional mapping}

Let us recall that a mapping $f:X\times X \rightarrow Y$ is biadditive provided
$$f(x+y,w)=f(x,w)+f(y,w),\quad f(x,z+w)=f(x,z)+f(x,w)  $$
for all $x,y,z,w\in X$.
El-Fassi \cite{FBC} investigated the stability of the following functional equation
\begin{eqnarray*}
f(x+y,z-w)+f(x-y,z+w)=2f(x,z)-2f(y,w).
\end{eqnarray*}
  The problem  is of profound significance for the stability of functional equations.

  Next, the type of the following function equation can be demonstrated by using a different method.
\begin{eqnarray*}
f(x+y,z-w)+f(x-y,z+w)=2f(x,z)-2f(y,w).
\end{eqnarray*}

\begin{theorem}\label{thm3.1}
Suppose that $X$ is a   linear normed space and  $Y$ is a Banach space. Let $f:X\times X\rightarrow Y$ be a mapping satisfying  \begin{eqnarray}\label{eqn3.1}
\|f(x+y,z-w)+f(x-y,z+w)-2f(x,z)+2f(y,w)\|
\leq  \|x\|^p\|y\|^p\|z\|^p\|w\|^p,
\end{eqnarray}with $p>3$, for all $x,y,z,w\in X$. If $f:X\times X\rightarrow Y$  is symmetric (i.e., $f(x,y)=f(y,x)$ for $x,y \in X$),
then there is a biadditive mapping  $K$ such that \begin{eqnarray*}\|K(x,y)-f(x,y)\|\leq \frac{1}{1-\left(3\left(\frac{3}{5}\right)^{2p}+2\left(\frac{4}{5}\right)^{2p}\right)}
{\left(\frac{12}{25}\right)^{p}\|x\|^{2p}\|y\|^{2p}} \end{eqnarray*}
for all $(x, y) \in X\times X$.
\end{theorem}

\begin{proof}
From the inequality (\ref{eqn3.1}), we  get
\begin{eqnarray}\label{eqn25'}
\left\|f(x,y)+ f\left(\frac{-x}{5},3y\right)+2 f\left(
\frac{3x}{5},y\right)-2 f\left(
\frac{2x}{5},2y\right)\right\|\leq  \left(\frac{12}{25}\right)^{p}\|x\|^{2p}\|y\|^{2p}
\end{eqnarray}
for all $x,y\in X$.

 Consider  $T:Y^{X\times X}\rightarrow Y^{X\times X}$ and $\mu :X\times X\rightarrow \mathbb{R}_+$ given as \begin{eqnarray*}
 T\xi(x,y)=- \xi\left(\frac{-x}{5},3y\right)+2 \xi\left(
\frac{3x}{5},y\right)-2\xi\left(
\frac{2x}{5},2y\right)
 \end{eqnarray*}for all $x,y\in X, \xi\in Y^{X\times X}$
 and
 $$\mu(x,y)
 ={\left(\frac{12}{25}\right)^{p}\|x\|^{2p}\|y\|^{2p}},\quad x\in X.$$
The inequality (\ref{eqn25'}) becomes
\begin{eqnarray*}
\|Tf(x,y)-f(x,y)\|\leq \mu(x,y), \forall x,y\in X.
\end{eqnarray*}
 For every $g,h\in Y^{X\times X}$, $x,y\in X$,
 \begin{eqnarray*}
 \begin{split}
 \;&\|Tg(x,y)-Th(x,y)\|\\
 \;& =
\left\|2 g\left(\frac{2x}{5},2y\right)- g\left(
 \frac{-x}{5},3y\right)-2g\left(
 \frac{3x}{5},y\right)\right.
 \\
 \;& \left. -2 h\left(\frac{2x}{5},2y\right)+ h\left(
 \frac{-x}{5},3y\right)+2h\left(
 \frac{3x}{5},y\right)
 \right\|\\
 \;&\leq 2\left\|g\left(\frac{3x}{5},y\right)
 -h\left(\frac{3x}{5},y\right)\right\|\\
 \;& +\left\|g\left(
 \frac{-x}{5},3y\right)
 -h\left(
 \frac{-x}{5},3y\right)\right\|
 +2\left\|g\left(\frac{2x}{5},2y\right)
 -h\left(\frac{2x}{5},2y\right)\right\|.
 \end{split}
 \end{eqnarray*}
So $T$ satisfies the inequality (\ref{eqn2.3})
with$f_1(x,y)=\left(\frac{3x}{5},y\right);
f_2(x,y)=\left(-\frac{x}{5},3y\right); f_3(x,y)=\left(-\frac{2x}{5},2y\right)$ and $\alpha_1=\alpha_3=2, \alpha_2 =1$.
Next, we define the operator $\Lambda:\mathds{R}_+^X\rightarrow \mathds{R}_+^X$, which  is given by
\begin{eqnarray*}
\Lambda \eta (x,y)=2\eta\left(\frac{3x}{5},y\right)
+\eta\left(-\frac{x}{5},3y\right)
+2\eta\left(\frac{2x}{5},2y\right)
\end{eqnarray*} for all $x,y\in X$. In particular,
\begin{eqnarray*}
\begin{split}
\;\Lambda \mu(x,y)&=2\mu\left(\frac{3x}{5},y\right)
+\mu\left(-\frac{x}{5},3y\right)+2\mu\left(\frac{2x}{5},2y\right)
\\
\;&=\left(3\left(\frac{3}{5}\right)^{2p}+2\left(\frac{4}{5}\right)^{2p}\right)\mu(x).
\end{split}
\end{eqnarray*}
Since $\Lambda$ is linear, we can get
\begin{eqnarray*}
\Lambda^n \mu(x)=\left(3\left(\frac{3}{5}\right)^{2p}+2\left(\frac{4}{5}\right)^{2p}\right)^n\mu(x),
x\in X,n\in \mathds{N}_0.
\end{eqnarray*}

Since  $\left(3\left(\frac{3}{5}\right)^{2p}+2\left(\frac{4}{5}\right)^{2p}\right)<1$,
the series $\sum_{n=0}^\infty \Lambda^n \mu (x)$ is convergent for every $x\in X$ and
\begin{eqnarray*}
\begin{split}
\;&\mu^*(x)=\sum_{n=0}^\infty \Lambda^n \mu(x)=\sum_{n=0}^\infty \left(3\left(\frac{3}{5}\right)^{2p}+
2\left(\frac{4}{5}\right)^{2p}\right)^n \mu(x)\\
\;&=\frac{1}{1-\left(3\left(\frac{3}{5}\right)^{2p}+2\left(\frac{4}{5}\right)^{2p}\right)}\mu(x)\\
\;&=\frac{1}{1-\left(3\left(\frac{3}{5}\right)^{2p}+2\left(\frac{4}{5}\right)^{2p}\right)}
{\left(\frac{12}{25}\right)^{p}\|x\|^{2p}\|y\|^{2p}}, x\in X.
\end{split}\end{eqnarray*}
By Theorem \ref{thm2.1}, there exists a mapping  $K:X\rightarrow Y$ such that
\begin{eqnarray*}
\begin{split}
\;& K(x,y)=\lim_{n\rightarrow \infty} T^n f(x,y)\\
\;& K(x,y)=- K\left(\frac{-x}{5},3y\right)+2 K\left(
\frac{3x}{5},y\right)-2K\left(
\frac{2x}{5},2y\right).
\end{split}
\end{eqnarray*}

By \cite[Property 2.2]{FBC},  $K$ satisfies the biadditive functional  equation.
\end{proof}

\subsection{The functional inequality}

 In this subsection, we solve and investigate the biadditive $\rho$-functional inequality in  normed spaces, which as  the following  form
\begin{eqnarray*}
\begin{split}
\;&\|f\left(x+y,z-w\right)+af\left(\frac{x-y}{a},z+w\right)-2f(x,z)+2f(y,w)   \|\\
\;&\leq \left\| \rho\left(f(x+y,z-w)+f(x-y,z+w)-2f(x,z)+2f(y,w)\right)\right\|.
\end{split}
\end{eqnarray*}

\begin{proposition}
Suppose that $X$ is a   linear normed space,  $Y$ is a Banach space, and $a\in \mathds{R}$, $a\neq 1$. Let $f:X\times X\rightarrow Y$ be a mapping satisfying
\begin{eqnarray}\label{eqn3.3}
\begin{split}
\;&\|f\left(x+y,z-w\right)+af\left(\frac{x-y}{a},z+w\right)-2f(x,z)+2f(y,w)   \|\\
\;&\leq \left\| \rho\left(f(x+y,z-w)+f(x-y,z+w)-2f(x,z)+2f(y,w)\right)\right\|.
\end{split}
\end{eqnarray}
for all $x,y,z,w\in X$.  If  $\rho<\frac{2}{5}$, then $ f:X\times X\rightarrow Y$ is biadditive.
\end{proposition}

\begin{proof}
Suppose that $f:X\times X\rightarrow Y$ satisfies (\ref{eqn3.3}). Letting $y=x=w=z=0$ in (\ref{eqn3.3}), we can get
$f(0,0)=0$. Similarly, assume that $x=y=w=0$ in (\ref{eqn3.3}), $f(0,x)=0$ can be obtained. By the same way, we can have$f(x,0)=0$.

Next, if  $y=x, w=0$ in (\ref{eqn3.3}), then
$$\|f(2x,z)-2f(x,z)\|\leq |\rho|\|f(2x,z)-2f(x,z)\|$$
for all $x,z\in X$. Thus $f(2x,z)=2f(x,z)$.  The same technique is applied to obtain
$af\left(\frac{x}{a},z\right)=f(x,z)$.

Letting $w=0$ in (\ref{eqn3.3}), we get
 \begin{eqnarray*}
 \begin{split}
 \;&
 \|f(x+y,z)+f(x-y,z)-2f(x,z)\|\\
 \;& =\left\|f\left(x+y,z\right)+af\left(\frac{x-y}{a},z\right)
 -2f\left(x,z\right)\right\|\\
 \;&\leq |\rho|\|f(x+y,z)+f(x-y,z)-2f(x,z)\|,
 \end{split}\end{eqnarray*}
ans so $f(x+y,z)+f(x-y,z)=2f(x,z)$ for all $x,y,z\in X$. Next, we can obtain that
$f:X\times X\rightarrow Y$ is additive in the second variable. Therefore $f(x,y)$ is biadditive.
\end{proof}

\begin{theorem}\label{thm3.2}
Suppose that $X$ is a   linear normed space,  $Y$ is a Banach space, and $a,\rho\in\mathds{R}$ with $2\rho\leq |1+a|$. Let $f:X\times X\rightarrow Y$ be a mapping satisfying
\begin{eqnarray}\label{eqn3.4}
\begin{split}
 \;&\left\|f\left(x+y,z-w\right)
 +af\left(\frac{x-y}{a},z+w\right) -2f(x,z)+2f(y,w)\right\|\\
 \;&\leq \left\|\rho\left(f(x+y,z-w)+f(x-y,z+w)-2f(x,z)+2f(y,w)\right)\right\|
 \\
 \;&+\left(\|x\|^r +\|y\|^r+\|z\|^r+\|w\|^r\right),
 \end{split}\end{eqnarray} for all $x,y,z,w\in X$. If $r>2$,  then there exists a unique biadditive mapping $K: X\rightarrow
Y$ such that
\begin{eqnarray*}
\|f(x,z)-K(x,z)\|\leq \frac{2^{3+r}}{2^r-1}\frac{\|x/2\|^r+\|z/2\|^r}{1-|\rho|}
, x,z\in X.
\end{eqnarray*}
\end{theorem}

\begin{proof}
Letting $x=y=z=w=0$  in (\ref{eqn3.4}), we get $f(0,0)=0$.
Letting $y=x,w=-z$ in (\ref{eqn3.4}), we get
\begin{eqnarray}\label{eqn3.5}
\begin{split}
\;&\|f(2x,2z)+af(0,0)-2f(x,z)+2f(x,-z)\|\\
\;&\leq
\|\rho(f(2x,2z)+f(0,0)-2f(x,z)+2f(x,-z))\|+2\|x\|^r+2\|z\|^r,
\end{split}\end{eqnarray}for all $x,z\in X$.

It follows from (\ref{eqn3.5}) that
\begin{eqnarray}\label{eqn3.6}
\|f(x,z)-2f(x/2,z/2)+2f(x/2,-z/2)\|\leq \frac{2\|x/2\|^r+2\|z/2\|^r}{1-|\rho|}.
\end{eqnarray}

 Consider  $T:Y^X\rightarrow Y^X$ and $\mu :X\times X\rightarrow \mathbb{R}_+$ given as \begin{eqnarray}
 Tf(x,z)=2f\left(x/2,z/2\right)
 -2f\left(x/2,-z/2
\right)
 \end{eqnarray}for all $x,z\in X, f\in Y^{X\times X}$
 and
 $$\mu(x,z)
 =\frac{2\|x/2\|^r+2\|z/2\|^r}{1-|\rho|},\quad x,z\in X.$$
The inequality (\ref{eqn3.6}) becomes
\begin{eqnarray}
\|(Tf)(x,z)-f(x,z)\|\leq \mu(x,z), \forall x,z\in X.
\end{eqnarray}
 For every $g,h \in Y^{X\times X}$ and $x,z\in X$,
 \begin{eqnarray}
 \begin{split}
 \;& \|Tg(x,z)-Th(x,z)\|\\
 \;&=\left\|2g\left(x/2,z/2\right)
 -2g\left(x/2,-z/2\right)
  -2h\left(x/2,z/2\right)
 +2h\left(x/2,-z/2\right)\right\|\\
 \;&\leq 2\left\|g\left(x/2,z/2\right)
 -h\left(x/2,z/2\right)\right\| +2\left\|g\left(x/2,-z/2\right)
 -h\left(x/2,-z/2\right)\right\|.
 \end{split}
 \end{eqnarray}
Let$f_1(x,z)=(x/2,z/2), f_2(x,z)=(x/2,-z/2)$ and $a_1(x)= a_2(x)=2$ in Remak \ref{rk}.
The operator $\Lambda:\mathds{R}_+^{X\times X}\rightarrow \mathds{R}_+^{X\times X}$ is given by
\begin{eqnarray}
\Lambda \eta (x,z)=2\eta\left(x/2,z/2\right)
+2\eta\left(x/2,-z/2\right)
\end{eqnarray} for all $x,z\in X$. In particular,
\begin{eqnarray}
\begin{split}
\;&\Lambda \mu(x,z)=2\mu\left(x/2,z/2\right)
+2\mu
\left(x/2,-z/2\right)\\
\;&=2\frac{2\|x/4\|^r+2\|z/4\|^r}{1-|\rho|}
+2\frac{2\|x/4\|^r+2\|z/4|^r}{1-|\rho|}
\\
\;&=\left(\frac{4}{2^r}\right)\mu(x).
\end{split}
\end{eqnarray}
Since $\Lambda$ is linear, we can get
\begin{eqnarray}
\Lambda^n \mu(x,z)=\left(\frac{4}{2^r}\right)^n\mu(x,z),
x,z\in X,n\in \mathds{N}_0.
\end{eqnarray}

Since  $\left(\frac{4}{2^r}\right)<1$,
the series $\sum_{n=0}^\infty \Lambda^n \mu (x)$ is convergent for every $x\in X$ and
\begin{eqnarray*}
\begin{split}
\;&\mu^*(x)=\sum_{n=0}^\infty \Lambda^n \mu(x)=\sum_{n=0}^\infty \left(\frac{4}{2^r}\right)^n \mu(x)\\
\;&=\frac{4}{1-\frac{1}{2^r}}\mu(x,z)\\
\;&=\frac{2^{3+r}}{2^r-1}\frac{\|x/2\|^r+\|z/2\|^r}{1-|\rho|}
, x,z\in X.
\end{split}\end{eqnarray*}
By Theorem \ref{thm2.1}, there exists a mapping $K:X\rightarrow Y$ such that
\begin{eqnarray*}
\begin{split}
\;& K(x,z)=\lim_{n\rightarrow \infty} T^n f(x,z)\\
\;& K(x,z)=2K\left(\frac{x}{2},\frac{z}{2}\right)-
2K\left(\frac{x}{2},-\frac{z}{2}\right)\\
\;& \|f(x,z)-K(x,z)\|\leq \frac{4}{1-\frac{1}{2^r}}\mu(x,z)\\
\;&=\frac{2^{3+r}}{2^r-1}\frac{\|x/2\|^r+\|z/2\|^r}{1-|\rho|}
, x,z\in X.
\end{split}
\end{eqnarray*}

Next, we prove that $K$ satisfies the biadditive functional equation.
Considering (\ref{eqn3.4}), we get
\begin{eqnarray}\label{eqn29}
\left\|f(x+y)+f(x-y)-Af\left(ax\right)
-Bf\left(by\right)-Df\left(-y\right)\right\|
\leq \theta \left(\left\|x\right\|^p +\left\|y\right\|^p   \right)
\end{eqnarray}for all $x,y\in X$.
Then, from the above inequality, we get
\begin{eqnarray}\label{eqn3.12}
\begin{split}
\;&\|f(x+y,z-w)+f(x-y,z+w)-2f(x,z)+2f(y,w)\|\\
\;&\leq 2\left(\left\|x\right\|^r +\left\|z\right\|^r   \right)+\|y\|^r+\|w\|^r, x,y,z,w\in X.\end{split}
\end{eqnarray}
 Hence
 \begin{eqnarray}
 \begin{split}
 \;&\|(Tf)(x+y,z-w)+(Tf)(x-y,z+w)-2(Tf)(x,z)+2(Tf)(y,w)\|\\
 \;&\leq \frac{4}{2^r}\frac{2\|x\|^r+\|y\|^r+2\|z\|^r+\|w\|^r}{1-|\rho|}
 \end{split}
 \end{eqnarray}
and so
\begin{eqnarray}
\begin{split}
\;&\|J^nf(x+y,z-w)+J^nf(x-y,z+w)-2J^nf(x,z)+2J^nf(y,w)\|\\
\;&\leq (\frac{4}{2^r})^n\frac{2\|x\|^r+\|y\|^r+2\|z\|^r+\|w\|^r}{1-|\rho|}
\end{split}\end{eqnarray}for all $n\in \mathds{N}_0$ and $x,y\in X$. Letting
$n\rightarrow \infty$, we obtain
\begin{eqnarray}
K(x+y,z-w)+K(x-y,z+w)-2K(x,z)+2K(y,w)=0, x,y\in X.
\end{eqnarray}
As a result, $K(x,y)$ is a biadditive mapping.
 \end{proof}

\section{Conclusion}

In the paper,  we examined the stability of the biadditive function equation of the form
\begin{eqnarray*}
f(x,y)=\alpha f(f_1(x,y))+\beta f(f_2(x,y)).
\end{eqnarray*}
The employed approach is utilized to examine several functional equations that involve multiple variables, and stability theorems have been  derived. Notably, this method does not impose any constraints on the parity, domain, or range of the function under scrutiny, unlike other techniques. We provided two examples to demonstrate the effectiveness of our approach.

\medskip

\medskip

\section*{Declarations}

\medskip

\noindent \textbf{Availablity of data and materials}\newline
\noindent Not applicable.

\medskip

\noindent \textbf{Conflict of interest}\newline
\noindent The authors declare that they have no competing interests.

\medskip

\noindent \textbf{Fundings}\newline
\noindent This work was supported by the National Natural Science
Foundation of China (no. 11761074), Project of Jilin Science
and Technology Development for Leading Talent of Science
and Technology Innovation in Middle and Young and Team
Project (no. 20200301053RQ), Natural Science Foundation
of Jilin Province (no. YDZJ202101ZYTS136), and Scientifc
Research Project of Guangzhou College of Technology and
Business in 2023 (no. KYPY2023012).

\medskip

\noindent \textbf{Authors' contributions}\newline
\noindent The authors equally conceived of the study, participated in its
design and coordination, drafted the manuscript, participated in the
sequence alignment, and read and approved the final manuscript.

\medskip

\noindent \textbf{Acknowledgements}\newline
\noindent  Not applicable.

\medskip

\bibliographystyle{amsplain}

\end{document}